\documentclass[12pt]{amsart}
\usepackage{fullpage,url,amssymb,enumerate,colonequals}
\usepackage[normalem]{ulem}
\usepackage[utf8]{inputenc}

\usepackage{mathrsfs} 
\usepackage{MnSymbol}
\usepackage{extarrows}
\usepackage{lscape}
\usepackage[all,cmtip]{xy}

\usepackage{amscd}

\usepackage[OT2,T1]{fontenc}

\usepackage{color}

\usepackage[
        colorlinks, citecolor=darkgreen,
        backref,
        pdfauthor={Imin Chen}, 
]{hyperref}

\usepackage{comment}

\newcommand{\F}{\mathbb{F}}

\newcommand{\Q}{\mathbb{Q}}
\newcommand{\R}{\mathbb{R}}

\newcommand{\rhobar}{{\overline{\rho}}}

\newcommand{\calO}{\mathcal{O}}

\newcommand{\Fq}{\mathfrak{q}}

\DeclareMathOperator{\sgn}{sgn}

\newcommand{\GL}{\operatorname{GL}}

\newcommand{\legendre}[3][]{\left(\frac{#2}{#3}\right)_{#1}}

\newcommand{\legendreL}[2]{\legendre[L]{#1}{#2}}

\newcommand{\inftyS}[1]{S_{#1}^\infty}
\newcommand{\evenS}[1]{S_{#1}^{\mathrm{even}}}
\newcommand{\oddS}[1]{S_{#1}^{\mathrm{odd}}}

\numberwithin{equation}{section}

\newtheorem{theorem}[equation]{Theorem}
\newtheorem{lemma}[equation]{Lemma}
\newtheorem{corollary}[equation]{Corollary}
\newtheorem{proposition}[equation]{Proposition}

\theoremstyle{definition}
\newtheorem{definition}[equation]{Definition}

\theoremstyle{remark}
\newtheorem{remark}[equation]{Remark}

\definecolor{darkgreen}{rgb}{0,0.5,0}

\title{Fermat's Last Theorem over $\Q(\sqrt{5})$ and $\Q(\sqrt{17})$}

\author{Imin Chen}

\address{Department of Mathematics, Simon Fraser University\\
Burnaby, BC V5A 1S6, Canada } \email{ichen@sfu.ca}

\author{Aisosa Efemwonkieke}

\address{Department of Mathematics, Simon Fraser University\\
Burnaby, BC V5A 1S6, Canada } \email{aisosa\_efemwonkieke@sfu.ca}

\author{David Sun}

\address{Department of Mathematics, Simon Fraser University\\
Burnaby, BC V5A 1S6, Canada } \email{david\_sun\_2@sfu.ca}

\date{\today}

\thanks{This work was supported by a NSERC Discovery Grant (I.C.) and NSERC USRA (A.E.).}

\begin{document}

\begin{abstract}
  We prove Fermat's Last Theorem over $\Q(\sqrt{5})$ and $\Q(\sqrt{17})$ for prime exponents $p \ge 5$ in certain congruence classes modulo $48$ by using a combination of the modular method and Brauer-Manin obstructions explicitly given by quadratic reciprocity constraints. The reciprocity constraint used to treat the case of $\Q(\sqrt{5})$ is a generalization to a real quadratic base field of the one used by Chen-Siksek. For the case of $\Q(\sqrt{17})$, this is insufficient, and we generalize a reciprocity constraint of Bennett-Chen-Dahmen-Yazdani using Hilbert symbols from the rational field to certain real quadratic fields.
\end{abstract}

\subjclass[2010]{Primary: 11D41; Secondary: 11F80, 11F03}

\maketitle

\tableofcontents

\section{Introduction}

The celebrated Fermat's Last Theorem was proven in \cite{Wiles} \cite{Taylor-Wiles} by Wiles and Taylor-Wiles. Ever since then, it has been natural to attempt to use the same methods to tackle more general forms of the Fermat equation, particularly the still unresolved Beal conjecture \cite{beal}.

Another direction has been to study the usual Fermat equation over number fields. For instance, Freitas and Siksek considered in \cite{freitas-siksek-1} and \cite{freitas-siksek-2}, Fermat's Last Theorem over real quadratic fields. In particular, it is currently known that Fermat's Last Theorem is asymptotically true for $K = \Q(\sqrt{d})$ where $2 \le d \le 23$ squarefree and $d \not= 5, 17$ and also on a subset of $d$ which has density $5/6$ among squarefree $d > 0$. For a given number field $K$, asymptotically true means that Fermat's Last Theorem is true for sufficiently large exponents. In a different vein, \cite{FKS-1} \cite{FKS-2} \cite{FKS-3} establish more general criteria for proving the asymptotic Fermat's Last Theorem, and apply these to a number of infinite families of number fields.

In this paper, we study Fermat's Last Theorem over $K$ for $d = 5$ and $d = 17$, the first two notable cases where asymptotic results are not yet proven. These cases present difficulties within the framework of \cite{freitas-siksek-1} due to a large number of solutions to the $S$-unit equation over $K$ where $S$ is the set of primes of $K$ above $2$. We circumvent these obstructions by showing the reciprocity constraints in \cite{chen-siksek} \cite{BCDY} generalize to a real quadratic base field, allowing a complete resolution in certain congruence classes of prime exponents $p$. In addition to \cite{chen-siksek} \cite{BCDY}, we also mention \cite{deconinck} \cite{ibrahim} where reciprocity constraints have been used to solve generalized Fermat equations.

Let $K$ be a quadratic field and $\calO_K$ its ring of integers. Assume that $K$ has class number one. We say that a solution $(a,b,c)$ over $\calO_K$, i.e.\ where $a,b,c \in \calO_K$, to
\begin{equation}
    \label{fermat-equation}
    x^p + y^p + z^p = 0 
\end{equation}
is primitive if the ideal $(a,b,c) = \calO_K$ and non-trivial if $abc \not= 0$. 

Under the assumption that $K$ has class number one, we note that any non-zero solution $(a,b,c) \in K^3$ to \eqref{fermat-equation} can be scaled to a primitive solution over $\calO_K$, and henceforth, when we refer to solutions over $K$, we mean that the solutions have been scaled to a primitive solution over $\calO_K$. In the more general case that $K$ does not have class number one, we refer the reader to \cite{freitas-siksek-1}.

\begin{theorem}
\label{main-5}
There are no non-trivial primitive solutions over $\Q(\sqrt{5})$ to 
\begin{equation*}
    x^p + y^p + z^p = 0 
\end{equation*}
for prime $p \ge 5$ if $p$ satisfies one of the following:
\begin{enumerate}
    \item $p \equiv 5,7 \pmod 8$, or
    \item $p \equiv 19, 41 \pmod {48}$.
\end{enumerate}
\end{theorem}
The above theorem shows that Fermat's Last Theorem over $\Q(\sqrt{5})$ is true for a set of prime exponents with Dirichlet density $5/8$. In \cite{kraussqrt5}, Theorem~\ref{main-5} is proven for $5 \le p < 10^7$, where we note the small exponents in this range rely on \cite{hao-parry}. The method used in \cite{kraussqrt5} is to fix a prime exponent $p \ge 11$ and use an auxiliary prime $q = pk+1$  with $k < p-2$ and $q$ split in $K$, and apply the modular method together with the local condition mod $q$ on solutions to \eqref{fermat-equation}. In practice, it appears most primes $q$ of this form succeed in giving a proof for the fixed exponent $p$, but there is no currently known way to prove results of this type in general.

\begin{theorem}
\label{main-17}
There are no non-trivial primitive solutions over $\Q(\sqrt{17})$ to 
\begin{equation*}
    x^p + y^p + z^p = 0 
\end{equation*}
for prime $p \ge 5$ if $p \equiv 5, 7 \pmod 8$.
\end{theorem}

\begin{remark}
In \cite{freitas-siksek-1}, Theorem~\ref{main-17} is established for $p \equiv 3, 5 \pmod 8$ using the symplectic method \cite{Halberstadt-Kraus}. The above theorem and the result in \cite{freitas-siksek-1} imply the corollary below, which shows that Fermat's Last Theorem over $K = \Q(\sqrt{17})$ is true for a set of prime exponents with Dirichlet density $3/4$.
\end{remark}

\begin{corollary}
\label{coro-17}
There are no non-trivial primitive solutions over $\Q(\sqrt{17})$ to 
\begin{equation*}
    x^p + y^p + z^p = 0 
\end{equation*}
for prime $p \ge 5$ if $p \not\equiv 1 \pmod{8}$.
\end{corollary}

The proof of the above theorems are based on the modular method, which attaches a Frey elliptic curve $E_0$ defined over $K$ to a putative non-trivial primitive solution. One then considers the representation $\rhobar_{E_0,p} : G_K \rightarrow \GL_2(\F_p)$ on the $p$-torsion points of $E_0$, which by virtue of being a Frey elliptic curve, has Artin conductor bounded independently of the solution and prime exponent $p$. 

Using modularity and level lowering, one deduces that $\rhobar_{E_0,p} \cong \rhobar_{f,p}$, where $\rhobar_{f,p} : G_K \rightarrow \GL_2(\F_p)$ is the residual Galois representation attached to a Hilbert newform at the possible Artin conductors. For the cases considered in this paper, $f$ has coefficient field equal to the field of rational numbers.

Unlike the original Fermat's Last Theorem, the space of newforms with level equal to the possible Artin conductor is typically not zero. Hence, to achieve a contradiction, one needs additional methods. To accomplish this, we combine information from reciprocity constraints such as in \cite{chen-siksek} to obtain a contradiction for prime exponents in certain congruence classes. This succeeds for $\Q(\sqrt{5})$, however the reciprocity constraint used in \cite{chen-siksek} is not strong enough to treat $\Q(\sqrt{17})$. 

Using the approach in \cite{BCDY}, valid over $\Q$, we prove a stronger reciprocity constraint, valid over certain real quadratic base fields, in terms of Hilbert symbols. This has an added advantage that the reciprocity law needed is simply the well-known reciprocity law for Hilbert symbols over a number field \cite{Voi20}. 

The programs and output transcripts for the computations needed in this paper are described and posted at \cite{programs}.

\section{Acknowledgements}

We would like to thank Alain Kraus for helpful comments and improvements on an earlier version of this paper, and Nuno Freitas and Samir Siksek for useful comments about the problems studied in this paper.

\section{Proof of Theorem~\ref{main-5}}

Let $K = \Q(\sqrt{5})$ and note that $\calO_K$ has unique factorization. Suppose $(a,b,c)$ is a non-trivial primitive solution over $\calO_K$ to $\eqref{fermat-equation}$ for a prime $p \ge 5$. Normalize the solution $(a,b,c)$ as in \cite{kraussqrt5}.

Let $E_0$ denote the Frey elliptic curve over $K$
\begin{equation}
    E_0: Y^2 = X(X-a^p)(X+b^p).
\end{equation}

Let $\mathfrak{P}$ be the unique prime of $K$ above $2$.

\begin{proposition}
\label{Conductor-even}
  Assuming $p \ge 5$, the conductor of $E_0$ over $K$ is given by
  \begin{equation*}
      N(E_0) = \mathfrak{P}^t \prod_{\mathfrak{q} \mid abc, \mathfrak{q} \not= \mathfrak{P}} \mathfrak{q}
  \end{equation*}
  where $t \in \{1,2,3\}$. Moreover, $t = 1$ if $2 \mid abc$ and $t \in \left\{ 2, 3 \right\}$ if $2 \nmid abc$.
\end{proposition}
\begin{proof}
  See \cite[Lemme 3]{kraussqrt5}.
\end{proof}

\begin{proposition}
  The representation $\rhobar_{E_0,p} : G_K \rightarrow \GL_2(\F_p)$ is irreducible if $p \ge 11$.
\end{proposition}
\begin{proof}
  See \cite[Proposition, p.\ 7]{kraussqrt5}.
\end{proof}

\begin{proposition}
\label{level-lower}
  There is a Hilbert newform $f$ of trivial character, parallel weight $2$, and level $\mathfrak{P}^t$ such that $\rhobar_{E_0,p} \simeq \rhobar_{f,\mathfrak{p}}$.
\end{proposition}
\begin{proof}
    The elliptic curve $E_0$ over $K$ is modular by \cite{freitas2014elliptic}. Using \cite[Theorem 7]{freitas-siksek-2}, we obtain the desired statement.
\end{proof}

The space of Hilbert newforms of trivial character, parallel weight $2$ and level $\mathfrak{P}, \mathfrak{P}^2$ is zero, so these cases do not occur. In particular, we may now assume that 
\begin{equation}
\label{odd-soln}
    2 \nmid abc 
\end{equation}
and $t = 3$.

There is a unique Hilbert newform $f$ of trivial character, parallel weight $2$, and level $\mathfrak{P}^3$ which corresponds to an elliptic curve $E$ over $K$. Hence, we have that 
\begin{equation}
\label{congruence}
  \rhobar_{E_0,p} \simeq \rhobar_{E,p}
\end{equation} 
by Proposition~\ref{level-lower}.

\begin{lemma}
The elliptic curve $E$ over $K$ is given by
\begin{equation*}
    E : Y^2 = X (X - (-8 + 4 \sqrt{5})) (X + (9 - 4 \sqrt{5})).
\end{equation*}
\end{lemma}
\begin{proof}
   The conductor of $E$ over $K$ is $\mathfrak{P}^3$. Since $E$ is modular, it corresponds to the unique Hilbert newform $f$ of trivial character, parallel weight $2$, and level $\mathfrak{P}^3$.
\end{proof}

\begin{remark}
In the last section, we explain how to determine the full list of solutions $(a,b,c)$ to $a + b + c = 0$, up to multiplication by a positive unit of $\calO_K$, such that $$E_{a,b,c} : Y^2 = X(X-a)(x+b)$$ gives rise to an elliptic curve in the same isogeny class of $E$ over $K$. We remark that the $2$-adic conditions on these $(a,b,c)$ vary, in particular, there are some triples where $2 \nmid abc$. This implies that inertia arguments at $2$ will fail as any solution which is $2$-adically close to one of these triples $(a,b,c)$ cannot be ruled out by inertia arguments at $2$.
\end{remark}

Case $r=1$:

Let $L = K(\zeta_r)$ where $\zeta_r$ is a primitive $r$th root of unity. Let $k = \mathcal{O}_K/3 \mathcal{O}_K \cong \mathbb{F}_9$, noting that $\# k^\times = 8$.

Let $\Fq_3 = 3 \mathcal{O}_K$ and $N(\Fq_3) = 9$ be the norm of $\Fq_3$. 

If \(abc = 0\) in \(k\), then we obtain a bound on $p$ by \cite[p.\ 9]{kraussqrt5}.
In particular, if $p \not= 3$, then $p$ divides
\begin{equation}
\label{Weil-bounds}
    a_{\Fq_3}(f) \pm (N(\Fq_3) +1) \not= 0
\end{equation}
which is non-zero by the Hasse bound 
\begin{equation}
  | a_{\Fq_3}(f)| = |a_{\Fq_3}(E)| \le 2 \sqrt{N(\Fq_3)}.
\end{equation}

If $a_{\Fq_3}(E_0) \not= a_{\Fq_3}(E)$, then we also obtain a bound on $p$ by \cite[p.\ 9]{kraussqrt5}.
In particular, if $p \not= 3$, then $p$ divides
\begin{equation}
\label{local-trace}
    a_{\Fq_3}(E_0) - a_{\Fq_3}(E) \not= 0.
\end{equation}
Thus, either we obtain a bound on $p$ or we have that
$a,b,c \in k^\times$ and $a_{\Fq_3}(E_0) = a_{\Fq_3}(E)$. The bound can be computed to be $p \in \left\{ 2,3 \right\}$.

Assume now that $a,b,c \in k^\times$ and $a_{\Fq_3}(E_0) = a_{\Fq_3}(E)$. Now set
\begin{equation}
    \epsilon = a^p b^p c^{-2p} \text{ in } k.
\end{equation}
Since $p^2 \equiv 1 \pmod 8$, we have that $\epsilon^R = ab c^{-2}$ in $k$ where $R = R^* \equiv p \pmod  8$. Hence, the condition \eqref{eq:reciprocity5} becomes 
\begin{equation}
\label{reciprocity-check-2}
    \legendre[K]{\epsilon^R-1}{3} \not= -1,
\end{equation}
for all permutations of $a,b,c$ as \eqref{odd-soln} holds. Using {\tt Magma}, we can check the set of triples $(a,b,c) \in (k^\times)^3$ which satisfy 
\begin{equation*}
  a_{\Fq_3}(E_0) = a_{\Fq_3}(E),
\end{equation*}
and \eqref{reciprocity-check-2} for all permutations of $a,b,c$ is empty if $p \equiv 5, 7 \pmod{8}$.

Case $r = 3$:

Let $L = K(\zeta_r)$ where $\zeta_r$ is a primitive $r$th root of unity. Let $k = \mathcal{O}_K/21 \mathcal{O}_K \cong \mathbb{F}_9 \times \mathbb{F}_{49}$, noting that $\# k^\times = 384$. Let $\mathfrak{q}_3 = 3 \calO_K$ and $\mathfrak{q}_7 = 7 \calO_K$.

If \(abc \notin k^\times \) or one of the following two conditions holds,
\begin{align*}
  a_{\Fq_3}(E_0) & \not= a_{\Fq_3}(E), \\
  a_{\Fq_7}(E_0) & \not= a_{\Fq_7}(E),
\end{align*}
we obtain a bound on $p$ similarly as in \eqref{Weil-bounds} and \eqref{local-trace} above. This bound can be computed to be $p \in \left\{ 2,3,5 \right\}$. The case $p = 5$ is covered by \cite{kraussqrt5}.

Assume from here on that $a,b,c \in k^\times$ and both
\begin{align*}
  a_{\Fq_3}(E_0) & = a_{\Fq_3}(E), \\
  a_{\Fq_7}(E_0) & = a_{\Fq_7}(E),
\end{align*}
hold. Now set
\begin{equation}
    \epsilon = a^p b^p c^{-2p} \text{ in } k.
\end{equation}
Suppose $p \equiv R^* \pmod{384}$ and let $R$ be such that $R R^* \equiv 1 \pmod{384}$. Then we have that $\epsilon^R = ab c^{-2}$ in $k$. Hence, the condition \eqref{eq:reciprocity5}, taking $\zeta_r' = \zeta_r^R$ since $r = 3$ divides $384$, becomes 
\begin{equation}
\label{reciprocity-check-1}
    \legendre[K]{\epsilon^R - \zeta_r'}{1- 4 \zeta_r} \not= -1,
\end{equation}
for all permutations of $a,b,c$ as \eqref{odd-soln} holds. Using {\tt Magma}, we can check the set of triples $(a,b,c) \in (k^\times)^3$ which satisfy 
\begin{align*}
  a_{\Fq_3}(E_0) & = a_{\Fq_3}(E), \\
  a_{\Fq_7}(E_0) & = a_{\Fq_7}(E),
\end{align*}
and \eqref{reciprocity-check-1} for all permutations of $a,b,c$ is empty if
\begin{align*}
  & p \equiv 7, 19, 29, 41, 55, 67, 77, 89, 103, 115, 125, 137, 151, 163, 173, 185, \\
  & 199, 211, 221, 233, 247, 259, 269, 281, 295, 307, 317, 329, 343, 355, 365, 377 \pmod{384}.
\end{align*}
It can be verified that the set of congruence condition above is equivalent to $p \equiv 7, 19, 29, 41 \pmod{48}$, noting that $48$ divides $384$. 

This concludes the proof of Theorem~\ref{main-5}.

\begin{remark}
  In the proof of Theorem~\ref{main-5}, we argued on a hypothetical solution $(a,b,c)$ to $a^p + b^p + c^p = 0$ and applied constraints over the residue class ring $k$. However, since $p$ is coprime to $\# k^\times$, we can check the modular and reciprocity constraints in terms of new variables $(a',b',c') = (a^p,b^p,c^p)$ in $k^3$. This remark also applies to Theorem~\ref{main-17}.
\end{remark}

\section{Reciprocity constraints}

In this section, we will prove the reciprocity constraint which is used in the proof of Theorem~\ref{main-5}. For this, we will use the following form of quadratic reciprocity over number fields and a corollary, both of which are taken from \cite{chen-siksek}. 

Let $L$ be a number field with ring of integers $\calO_L$. For an element or ideal of $\calO_L$, we say that it is odd if it is coprime to $2 \calO_L$.
\begin{theorem}\label{thm:genrec}
	Suppose \( L \) is a number field with \(r\) real embeddings. We write
	\(\sgn_{i}(\alpha)\) for the sign of the image of \(\alpha\) under the
		\(i\)th real embedding. Let \(\alpha, \lambda \in \calO_L\) be
		integers with \(\alpha\) odd. Decompose \(\lambda \calO_L =
		\mathfrak{L}\mathfrak{R}\) where \(\mathfrak{R}\) is an odd
		ideal in $\calO_L$. Suppose that \(\alpha\) is a quadratic residue modulo
		\(4\mathfrak{L}\). Then 
		\[
			\legendre[L]{\lambda}{\alpha} \legendre[L]{\alpha}{\mathfrak{R}}
			= {(-1)}^{\sigma}
		\]
		where $\legendre[L]{a}{\mathfrak{R}}$ is the Jacobi symbol in $L$ and 
		\[
			\sigma = \sum_{i=1}^{r} \frac{\sgn_{i}(\alpha) -
			1}{2}\frac{\sgn_{i}(\lambda) - 1}{2}.
		\]
\end{theorem}
For the definition of Jacobi symbol $\legendre[L]{\lambda}{\mathfrak{R}}$ over a number field $L$, see \cite{chen-siksek} or \cite[Definition 8.2]{neukirch}.
Also, for $\alpha \in \calO_L$, $\legendre[L]{\lambda}{\alpha} := \legendre[L]{\lambda}{\alpha \calO_L}$.

\begin{corollary}\label{cor:genrec-cor}
	Let \(\alpha, \lambda\) be algebraic integers in the number field \( L \) with
	\(\alpha\) odd. Suppose that \(\alpha \equiv \epsilon^{2}
	\pmod{4\lambda}\) for some algebraic integer \(\epsilon\) in $L$. In addition, suppose
	that \(\alpha\) is positive in every real embedding of \( L \). Then 
	\[
		\legendre[L]{\lambda}{\alpha} \neq -1.
	\]
\end{corollary}

We also require the following lemmas.

\begin{lemma}
\label{fermat-real-r1}
\label{fermat-real}
If $x, y \in \mathbb{R}$ satisfy $x^n + y^n = 1$ for $n \in \mathbb{N}$, then $xy < 1$.
\end{lemma}
\begin{proof}
For all real $x$, we have that \(x^{2n} - x^{n} + 1 > 0\). This implies that \(x^{n}(1 - x^{n}) < 1\). If $xy \ge 1$, then we would have that
$x^n y^n = x^n (1-x^n) \ge 1$, a contradiction. Hence we conclude $xy < 1$.
\end{proof}

\begin{lemma}
\label{totally-neg}
		If \((a, b, c)\) is a primitive non-trivial solution to
		Fermat's equation over K for some odd prime \(p >
		2\), then \(ab - c^{2}\) is negative in every real embedding of
		\(K\).
	\end{lemma}
	\begin{proof}
		If \((a, b, c)\) is such a triple then
		\[
			{\bigg(\frac{a}{-c}\bigg)}^{p} +
			{\bigg(\frac{b}{-c}\bigg)}^{p} = 1.
		\]
		Now we employ Lemma~\ref{fermat-real} to deduce
		\[
			\frac{ab}{c^{2}} < 1.
		\]
		which gives \(ab - c^2 < 0\).
		Now let \(\rho\) be the other embedding of \(K\) into \(\R\).
		Then \((\rho a, \rho b, \rho c)\) is also a primitive non-trivial solution
		and so using the same argument gives that
		\[
			\rho(ab - c^{2}) = (\rho{a})(\rho{b}) - {(\rho{c})}^2 <
			0.
		\]
	\end{proof}

We now state and prove the reciprocity constraint.
\begin{theorem}\label{thm:sqrt5recip}\label{first-claim}
	Let \(K\) be a number field and let \(\zeta_{r}\) be a primitive $r$th root of unity where \( (r,p) = 1 \). If $r = 1$, assume further that $K$ has an even number of real embeddings. Let \(L = K(\zeta_{r})\) and write \(\zeta_r = {\zeta'_{r}}^p \) where $\zeta'_r$ is a primitive $r$th root of unity. 
	
	If $(a,b,c)$ is a primitive solution over $\calO_K$ to 
	\[
		a^{p} + b^{p} + c^{p} = 0
	\]
	and \(c\) is coprime to \(2\calO_{K}\), then
	\begin{equation}\label{eq:reciprocity5}
				\legendreL{ab - c^{2}\zeta'_{r}}{1 - 4\zeta_{r}}
				\neq -1.
	\end{equation}
	In particular, if $abc$ coprime to $2 \calO_K$, then \eqref{eq:reciprocity5} holds for all permutations of $a,b,c$.
\end{theorem}

\begin{proof}
    By assumption, \(c\) is coprime to \(2\calO_{K}\). We employ the identity
    \[
        (a^{p} - b^{p})^{2} = -4{(ab)}^{p} + c^{2p}.
    \]
    Subtracting both sides of this identity by \(c^{2p}(1 - 4\zeta_{r})\) we get
    \begin{align}\label{eq:claim-fac}
		{(a^{p} - b^{p})}^{2} - c^{2p}(1 - 4\zeta_{r})
		& = -4 \left( (ab)^p - \zeta_r c^{2p} \right) \\
		\notag & = -4(ab -c^{2}\zeta'_{r})h 
	\end{align}
	where
	\begin{equation*}
		h = {(ab)}^{p-1} + \zeta'_{r}c^{2}{(ab)}^{p-2} + \ldots +
		\zeta_{r}^{\prime(p-1)}c^{2(p - 1)}.
	\end{equation*}
	Let \(\alpha = 1 - 4\zeta_{r}\), \(\lambda = ab - c^{2} \zeta'_{r}\).
	We first show that \(c\) is invertible modulo \(4\lambda = 4(ab -
	c^{2}\zeta'_{r})\). To this end, let \(\mathfrak{P}\) be a prime in
	\(\calO_{L}\) dividing \(c\). 
	We show that \(\mathfrak{P}\) does not divide \(4\lambda\).
	Indeed, suppose this were the case. Then \(\mathfrak{P}\) divides
	\(4\lambda\) and so either divides \(4 \calO_L\) or divides \(\lambda\). Note
	that \(\mathfrak{P}\) cannot divide \(4 \calO_L\), for then 
	\(\mathfrak{P}\) divides \(2 \calO_L\) and thus 
	\(c \in \mathfrak{P} \cap \calO_{K}\), contradicting that $c$ is coprime to $2 \calO_K$. Hence
	\(\mathfrak{P} \mid \lambda\). Then we have that
	\[
		ab = \lambda + c^{2}\zeta'_{r} \in \mathfrak{P}
	\]
	and so either \(a\) or \(b\) is in \(\mathfrak{P}\). But
	\[
		a^{p} + b^{p} + c^{p} = 0
	\]
	so \(\mathfrak{P}\) divides \(a, b\) and \(c\). We have assumed that
	\(a, b, c\) are coprime in \(\calO_{K}\). Let \(r_{1}, r_{2}, r_{3}\) be
	such that 
	\[
		r_{1}a + r_{2}b + r_{3}c = 1
	\]
	Then we get the contradiction that \(1 \in \mathfrak{P}\).

	In light of this, let \(\epsilon \in \calO_{L}\) be such that the image of \(\epsilon\)
	in \(\calO_{L}/(4(ab - c^{2}\zeta'_{r}))\) is 
	\[
		c^{-p}{(a^{p} - b^{p})} 
		\pmod{4(ab - c^{2}\zeta'_{r})}.
	\]
	We see that \(\alpha\) is odd, for if $\mathfrak{P}$ divides both $\alpha$ and $2 \calO_L$, then $\mathfrak{P}$ divides $1$, a contradiction. Then from~\eqref{eq:claim-fac}
	\[
		\alpha \equiv \epsilon^{2} \pmod{4\lambda}.
	\]
	Since \(L\) has no real embeddings if \(r > 2\),
	Corollary~\ref{cor:genrec-cor} gives the result in this case. 
	
	If \(r = 2\), we have that \(\alpha = 5\) is positive in every real embedding of
	\(L = K\) because \(\Q\) is fixed. Corollary~\ref{cor:genrec-cor} then
	gives the result in this case.
	
	From Theorem~\ref{thm:genrec} we have that
	\[
		\legendre[K]{ab - c^{2}}{-3} = {(-1)}^{\sigma}.
	\]
	We are under the assumption that $K$ has an even number of real embeddings $2t$ so that 
	\[
		\sigma = \sum_{i=1}^{2t} \frac{\sgn_{i}(-3) -
		1}{2}\frac{\sgn_{i}(ab - c^{2}) - 1}{2}.
	\]
	We see that \(\sgn_{i}(-3) = -1\) for all \(i\). Lemma~\ref{totally-neg} gives that
	\( ab - c^{2} \) is totally negative and so \(\sgn_{i}(ab - c^{2}) = -1\)
	for all \(i\) and hence we conclude \(\sigma = 2t \) is even. This gives the result for $r = 1$.
\end{proof}

\section{Proof of Theorem~\ref{main-17}}

Let $K = \Q(\sqrt{17})$ and $\mathcal{O}_K$ its ring of integer which has unique factorization. The prime $2$ splits in $K$ as $2 \calO_K = \mathfrak{P}_1 \mathfrak{P}_2$ for prime ideals $\mathfrak{P}_1$ and $\mathfrak{P}_2$ of $K$. Suppose $(a,b,c)$ is a non-trivial primitive solution over $\mathcal{O}_K$ to $\eqref{fermat-equation}$ and let $E_0$ denote the Frey elliptic curve over $K$
\begin{equation}
    E_0: Y^2 = X(X-a^p)(X+b^p).
\end{equation}

\begin{proposition}
\label{Conductor-17}
  Assume $p \ge 5$. Up to scaling $(a,b,c)$ by a unit in $\calO_K$, the conductor of $E_0$ over $K$ is given by
  \begin{equation*}
      N(E_0) = \mathfrak{P}_1 \mathfrak{P}_2 \prod_{\mathfrak{q} \mid abc, \mathfrak{q} \nmid \mathfrak{P}_1 \mathfrak{P}_2} \mathfrak{q}.
  \end{equation*}
\end{proposition}
\begin{proof}
   See \cite[Corollary 5.1]{freitas-siksek-1}. We note the fact that $\calO_K$ has unique factorization means that there is no need to consider the extraneous prime $\mathfrak{r}$ in \cite{freitas-siksek-1}.
\end{proof}

We may assume that
\begin{equation}
\label{even-soln}
    2 \mid abc, 
\end{equation}
since $a^p + b^p + c^p = 0$ implies one of $a^p,b^p,c^p$ is $0$ in the residue fields of $\mathfrak{P}_1$ and $\mathfrak{P}_2$.

By the arguments in \cite[p.\ 2]{freitas-siksek-1}, we need only prove Theorem~\ref{main-17} for $p \ge 17$, which we now assume.
\begin{proposition}
  The Galois representation $\rhobar_{E_0,p} : G_K \rightarrow \GL_2(\F_p)$ is irreducible if $p \ge 17$.
\end{proposition}
\begin{proof}
 See \cite[Lemma 6.1., p.\ 9]{freitas-siksek-1}.
\end{proof}

\begin{proposition}
\label{level-lower-17}
  There is a Hilbert newform $f$ of trivial character, parallel weight $2$, and level $\mathfrak{P}_1 \mathfrak{P}_2$ such that $\rhobar_{E_0,p} \simeq \rhobar_{f,\mathfrak{p}}$.
\end{proposition}
\begin{proof}
    The elliptic curve $E_0$ over $K$ is modular by \cite{freitas2014elliptic}. Using \cite[Theorem 7]{freitas-siksek-2}, we obtain the desired statement.
\end{proof}

There is an unique Hilbert newform $f$ of trivial character, parallel weight $2$ and level $\mathfrak{P}_1\mathfrak{P}_2$ and this corresponds to an elliptic curve $E$ over $K$. 
\begin{lemma}
The elliptic curve $E$ over $K$ is given by
\begin{equation*}
    E : Y^2 = X (X - (4 - \sqrt{17})) \left( X + \frac{-13+5 \sqrt{17}}{2} \right).
\end{equation*}
\end{lemma}
\begin{proof}
   See \cite[p.\ 13]{freitas-siksek-1}. The conductor of $E$ over $K$ is $2 \calO_K$. Since $E$ is modular, it corresponds to the unique Hilbert newform $f$ of trivial character, parallel weight $2$, and level $2 \calO_K$.
\end{proof}
Hence, we have that 
\begin{equation}
\label{congruence-17}
  \rhobar_{E_0,p} \simeq \rhobar_{E,p}
\end{equation} 
by Proposition~\ref{level-lower-17}.

\begin{proof}[Proof of Theorem \ref{main-17}]

We give two proofs using $r = 1$ and $r = 3$.

Case $r=1$:

Let $L = K(\zeta_r)$ where $\zeta_r$ is a primitive $r$th root of unity. Let $k = \mathcal{O}_K/3 \mathcal{O}_K \cong \mathbb{F}_9$, noting that $\# k^\times = 8$. Let $\mathfrak{q}_3 = 3 \calO_K$. 

If \(abc \notin k^\times \) or $a_{\Fq_3}(E_0)  \not= a_{\Fq_3}(E)$, then we obtain a bound on $p$ similarly as in \eqref{Weil-bounds} and \eqref{local-trace} above. This bound can be computed to be $p \in \left\{ 2,3 \right\}$.

Assume from here on that $a,b,c \in k^\times$ and $a_{\Fq_3}(E_0) = a_{\Fq_3}(E)$ holds. For $a,b,c \in k^\times$, we set
\begin{equation}
    \epsilon = a^p b^p c^{-2p} \text{ in } k.
\end{equation}
Since $p^2 \equiv 1 \pmod 8$, we have that $\epsilon^R = ab c^{-2}$ in $k$ where $R = R^* \equiv p \pmod  8$. Hence, the condition \eqref{jacobi} becomes 
\begin{equation}
\label{reciprocity-check-3}
    \legendre[K]{\epsilon^R-1}{3} \not= -1,
\end{equation}
for all permutations of $a,b,c$ as \eqref{even-soln} holds. We note that \eqref{unity-condition} holds as if not, by the left hand side of \eqref{hilbert-identity}, we would have that $\sqrt{-3} \in K$.

Using {\tt Magma}, we can check the set of triples $(a,b,c) \in (k^\times)^3$ which satisfy 
\begin{equation*}
  a_{\Fq_3}(E_0) = a_{\Fq_3}(E),
\end{equation*}
and \eqref{reciprocity-check-3} for all permutations of $a,b,c$ is empty if $p \equiv 5, 7 \pmod{8}$.

Case $r = 3$:

Let $L = K(\zeta_r)$ where $\zeta_r$ is a primitive $r$th root of unity. Let $k = \mathcal{O}_K/21 \mathcal{O}_K \cong \mathbb{F}_9 \times \mathbb{F}_{49}$, noting that $\# k^\times = 384$. Let $\mathfrak{q}_3 = 3 \calO_K$ and $\mathfrak{q}_7 = 7 \calO_K$.

If \(abc \notin k^\times \) or one of the following two conditions holds,
\begin{align*}
  a_{\Fq_3}(E_0) & \not= a_{\Fq_3}(E), \\
  a_{\Fq_7}(E_0) & \not= a_{\Fq_7}(E),
\end{align*}
we obtain a bound on $p$ similarly as in \eqref{Weil-bounds} and \eqref{local-trace} above. This bound can be computed to be $p \in \left\{ 2,3,5,7 \right\}$.

Assume from here on that $a,b,c \in k^\times$ and both
\begin{align*}
  a_{\Fq_3}(E_0) & = a_{\Fq_3}(E), \\
  a_{\Fq_7}(E_0) & = a_{\Fq_7}(E),
\end{align*}
hold. For $a,b,c \in k^\times$, we set
\begin{equation}
    \epsilon = a^p b^p c^{-2p} \text{ in } k.
\end{equation}
Suppose $p \equiv R^* \pmod{384}$ and let $R$ be such that $R R^* \equiv 1 \pmod{384}$. Then we have that $\epsilon^R = ab c^{-2}$ in $k$. Hence, the condition \eqref{jacobi}, taking $\zeta_r' = \zeta_r^R$ since $r = 3$ divides $384$, becomes 
\begin{equation}
\label{reciprocity-check}
    \legendre[K]{\epsilon^R -  \zeta_r'}{1- 4 \zeta_r} \not= -1,
\end{equation}
for all permutations of $a,b,c$ as \eqref{even-soln} holds. We note that the hypothesis \eqref{unity-condition} holds as $K$ is totally real.

Using {\tt Magma}, we can check the set of triples $(a,b,c) \in (k^\times)^3$ which satisfy 
\begin{align*}
  a_{\Fq_3}(E_0) & = a_{\Fq_3}(E), \\
  a_{\Fq_7}(E_0) & = a_{\Fq_7}(E),
\end{align*}
and \eqref{reciprocity-check} for all permutations of $a,b,c$ is empty if 
\begin{align*}
& p \equiv 5, 7, 13, 23, 29, 31, 37, 47, 53, 55, 61, 71, 77, 79, 85, 95, 101, 103, \\
& 109, 119, 125, 127, 133, 143, 149, 151, 157, 167, 173, 175, 181, 191, 197, 199, 205, \\
& 215, 221, 223, 229, 239, 245, 247, 253, 263, 269, 271, 277, 287, 293, 295, 301, \\
& 311, 317, 319, 325, 335, 341, 343, 349, 359, 365, 367, 373, 383 \pmod{384}.
\end{align*}
It can be verified that the congruence condition above is equivalent to $p \equiv 5, 7, 13, 23 \pmod{24}$, noting that $24$ divides $384$. Finally, the congruence $p \equiv 5, 7, 13, 23 \pmod{24}$ is equivalent to $p \equiv 5, 7 \pmod 8$ for prime $p \ge 5$.

\end{proof}

\section{Reciprocity constraints using the Hilbert symbol}

In this section, we use Hilbert symbols to prove a strengthened reciprocity constraint which does not have a condition on $c$.
\begin{definition}
For a global field $L$ we define the \textit{Hilbert symbol} $( \cdot, \cdot)_L:L^{\times}\times L^{\times} \rightarrow \{-1,1\}$ as
\begin{equation}
\label{hilbertsymbol}
(a,b)_L := \begin{cases}1, &\text{ if $z^2 = ax^2 + by^2$ has a non-trivial solution in $L$,}\\
-1, &\text{ otherwise.}\end{cases}
\end{equation}
Let $S_L$ denote the set of normalized places of $L$ and partition them into
\begin{align*}
    S_L^\infty & = \left\{ v \in S_L : v \mid \infty \right\}, \\
    S_L^{\text{even}} & = \left\{ v \in S_L : v \mid 2 \right\}, \\
    S_L^{\text{odd}} & = \left\{ v \in S_L : v \nmid 2 \right\}.
\end{align*}

For a place $v \in S_L$ of $L$, we denote by $(\alpha,\beta)_v := (\alpha, \beta)_{L_v}$, where $L_v$ is the completion of $L$ at $v$. Let $\pi_v$ be a uniformizer for $L_v$, $\mathcal{O}_v$ the ring of integers of $L_v$, and $\F_v$ the residue field of $L_v$.
\end{definition}

We will state a few useful properties of the Hilbert symbol for later use.
\begin{lemma}
\label{hilbert-pairing}
The Hilbert symbol defines a nondegenerate symmetric bimultiplicative pairing.
\end{lemma}
\begin{proof}
  See \cite[Lemma 12.4.6]{Voi20}.
\end{proof}

\begin{lemma}
\label{hilbert-prop}
Let $a, b \in L^\times$. Then the following hold:
\begin{enumerate}
    \item $(a c^2,b d^2)_L = (a,b)_L$ for all $c, d \in L^\times$.
    \item $(b,a)_L = (a,b)_L$.
    \item $(a,b)_L = (a,-ab)_L = (b,-ab)_L$.
    \item $(1,a)_L = (a,-a)_L = 1$.
    \item If $a \not= 1$, then $(a, 1-a) = 1$.
    \item If $\sigma \in \text{Aut}(L)$, then $(a,b)_L = (\sigma(a),\sigma(b))_L.$
\end{enumerate}
\end{lemma}
\begin{proof}
  See \cite[Lemma 12.4.3]{Voi20}.
\end{proof}

\begin{lemma}[Reciprocity Law]
Let $L$ be a number field and $S_L$ the set of places of $L$. Then for any $\alpha, \beta \in L^*$, we have that
\begin{equation}\label{eq:hilbertrec}
    \prod_{v \in S_L} (\alpha, \beta)_v = 1
\end{equation}
\end{lemma}
\begin{proof}
  See \cite[Corollary 14.6.2]{Voi20}.
\end{proof}

\begin{lemma}
\label{hilbert-odd}
With notation as above, let $q = \# \F_v$ be odd. Write $a = a_0 \pi^{v(a)}$ and $b = b_0 \pi^{v(b)}$. Then we have that
\begin{equation}
    \label{eq:odd-hilbert}
    (a,b)_v = (-1)^{v(a) v(b) (q-1)/2} \genfrac(){}{}{a_0}{\pi_v}^{v(b)} \genfrac(){}{}{b_0}{\pi_v}^{v(a)}.
\end{equation}
\end{lemma}
\begin{proof}
  See \cite[(12.4.10)]{Voi20}.
\end{proof}

We now move on to the theorem that enables us to tackle the reciprocity constraint on the primitive solutions over $\mathbb{Q}(\sqrt{17})$. The following is a generalization of \cite[Proposition 17]{BCDY}.
\begin{theorem}
\label{thm:hilbert-constraint}
Let $L$ be a number field containing a primitive $r$th root of unity $\zeta_r$ such that $(r,n) = 1$ where $n \in \mathbb{N}$. Assume $s, t \in \calO_L$, $v(t) = 0$ for all $v \in S_L$ such that $v(n) > 0$, and  $v(s) > 0$ only for places $v \in
S_L^{\text{even}}$. Furthermore, suppose we have the following identity
\begin{equation}
\label{hilbert-identity}
    A^2 - t B^{2n} = s (C^n - \zeta_r B^{2n}) \not= 0
\end{equation}
for coprime $A, B, C \in \calO_L$ and write $\zeta_r = {\zeta_r'}^n$ where $\zeta_r'$ is a primitive $r$th root of unity. 

Then we have that
\begin{equation}
\label{reciprocity-product}
  \prod_{v \in S_L^\infty} (t,s (C - \zeta_r' B^2))_v \prod_{v \in S_L^{\text{even}}} (t, s (C - \zeta_r' B^2))_v \prod_{v \in S_L^{\text{odd}}, v(t) \text{ odd}} (t,s(C - \zeta_r' B^2))_v = 1.
\end{equation}
\end{theorem}
\begin{proof}
The hypotheses imply $s, t, (C^n - \zeta_r B^{2n})$ are non-zero, and hence also $(C - \zeta_r' B^2) \not= 0$, so the Hilbert symbols used below are well-defined. Note that
  \begin{equation*}
    (t, s(C^{n} - \zeta_{r}B^{2n}))_{v} = 1
  \end{equation*}
  for all \(v \in S_{L}\) because
  \begin{equation*}
  A^{2} = t B^{2n} + s(C^{n} - \zeta_{r} B^{2n})\cdot 1^{2}.   
  \end{equation*}
  Using the fact that $\zeta_r = {\zeta_r'}^n$ and the factorization
  \begin{align}
  \label{factor-ABC}
      C^{n} - \zeta_{r}B^{2n} &= C^n - (\zeta_r' B^2)^n\\
      \notag &= (C - \zeta_{r}' B^{2})(C^{n-1} + C^{n-2} \zeta_{r}' B^{2} + \cdots +  (\zeta_r' B^2)^{n-1})
  \end{align}
  we see that
  \begin{align}\label{eq:hilbert-factor-rec}
    (t, s(C - \zeta'_{r}B^{2}))_{v} = (t, C^{n-1} + C^{n-2} \zeta_{r}' B^{2} + \cdots + (\zeta_r' B^2)^{n-1} )_{v},
  \end{align}
  using Lemma~\ref{hilbert-pairing}.

  Let \(\beta = s(C - \zeta_{r}'B^{2})\).
  By~\eqref{eq:hilbertrec}, we have that
  \begin{equation}
  \label{use-hilbert-reciprocity}
    \prod_{v \in \inftyS{L}} (t, \beta)_{v}\prod_{v \in \evenS{L}} (t,
    \beta)_{v}\prod_{v \in \oddS{L}, v(t) \text{ odd}} (t, \beta)_{v} = \prod_{v \in
        \oddS{L}, v(t) \text{ even}} (t, \beta)_{v}.
    \end{equation}
    Thus it suffices to show that \((t, \beta)_{v} = 1\) when \(v \in
    \oddS{L}\) and \(v(t)\) is even.

    Suppose $v \in \oddS{L}$ and $v(t)$ is even. By~\eqref{eq:odd-hilbert},
    \((t, \beta)_{v} = 1\) when \(v(\beta)\) is even. So suppose that \(v(\beta)\) is odd. 
    
    If \(v(n) > 0\), then we have that \(v(t) = 0\) by assumption. As \( v(\beta) \) is odd, we deduce from \eqref{hilbert-identity} and \eqref{factor-ABC} that
    \begin{equation*}
     v(A^{2} - t B^{2n}) = v(s(C^{n} - \zeta_{r} B^{2n})) \ge v(\beta) > 0
   \end{equation*}
   so that 
   \begin{equation}
   \label{AB-cong}
     A^{2} \equiv t B^{2n} \pmod {\pi_{v}}.
   \end{equation}
   Since \(A, B, C\)
   are coprime, by \eqref{hilbert-identity}, we deduce that $A$ and $B$ are coprime, as $v(A), v(B) > 0$ implies $v(s C^n) > 0$. Since $v(s) = 0$, we see that $v(C) > 0$, contradicting \(A, B, C\) being coprime. Hence, if $v(B) > 0$, then by \eqref{AB-cong} we obtain that $v(A) > 0$, contradicting \(A, B \) being coprime. Thus, $v(B) = 0$ and $B$ is a $v$-adic unit. From \eqref{AB-cong}, we deduce that
   \begin{equation*}
     (A B^{-n})^2 \equiv t \pmod{\pi_v}.
   \end{equation*} 
   and hence
   \begin{equation*}
       \legendre{t}{\pi_v} = 1.
   \end{equation*}
   using also that $v(t) = 0$. Using Lemma~\ref{hilbert-odd}, this leads to \((t, \beta) = 1\). 
   
   If \(v(n) = 0\), then since \(v(\beta)\) is odd and $v(s) = 0$, we firstly have that 
   \begin{equation}
   \label{one-factor}
     v(C - \zeta_r' B^2) = v(\beta) > 0.
   \end{equation}
   We have that $v(C) = 0$ for if $v(C) > 0$, then by \eqref{one-factor} $v(B) > 0$ and hence by \eqref{hilbert-identity} $v(A) > 0$, contradicting that \(A, B, C\) are coprime. It follows that
   \begin{equation*}
       C^{n-1} + \ldots +  (\zeta_r' B^2)^{n-1} \equiv n C^{n-1} \pmod{\pi_v},
   \end{equation*}
   which from the conditions $v(n) = v(C) = 0$ imply that
   \begin{equation*}
     v(C^{n-1} + \ldots +  (\zeta_r' B^2)^{n-1}) = 0.
   \end{equation*}
   Since $v(t)$ is even, using Lemma~\ref{hilbert-odd}, it follows that \((t, C^{n-1} + \ldots + (\zeta_r' B^2)^{n-1})_{v}  = 1\). Hence, by \eqref{eq:hilbert-factor-rec} we obtain
   \begin{equation*}
     (t, \beta)_{v} = (t, C^{n-1} +  \ldots + (\zeta_r' B^2)^{n-1} )_{v} = 1,
  \end{equation*}
  as desired.
\end{proof}

\begin{remark}
\label{hilbert-apply}
In our application, $L = K(\zeta_r)$, $n = p$, $t = 1 - 4 \zeta_r$, $s = 4$, $C = ab, B = c, A = a^p-b^p$ and we have the identity \eqref{hilbert-identity} because of \eqref{eq:claim-fac}. Here, $v(s) > 0$ only for places $v \in
S_L^{\text{even}}$ and $A,B,C$ are coprime. Finally, we require $(p)$ to be coprime to  $(t) = (1 - 4 \zeta_r)$ to ensure the hypothesis $v(t) = 0$ for all $v \in S_L$ such
that $v(n) > 0$.
\end{remark}

We need a few more lemmas to aid the proof of our reciprocity constraint.
\begin{lemma}
\label{hilbert-even}
Let $L$ be a number field containing a primitive $r$th root of unity $\zeta_r$ and $L$ is unramified at $2$. Let $v$ be a place of $L$ above $2$. Then $(1 - 4 \zeta_r,b)_v = 1$  for any $b \in L$ with $v(b) = 0$.
\end{lemma}
\begin{proof}
  The proof in \cite[Theorem 1, Chapter III, \S 1.2]{Serre} for the case $L = \Q$ and $\alpha = \beta = 0$ can be adapted to prove this lemma.
  
  If \(u\) and \(b\) are elements in \(L_v^{\times}\) with \(u \equiv 1 \pmod{8}\), then we first show that
  \begin{equation*}
    (u, b)_v = 1
  \end{equation*}
  where the Hilbert symbol is taken in \(L_{v}\). Indeed if \(u \equiv 1 \pmod{8}\), then
  since 
  \begin{equation*}
    v(1^2 - u) \geqslant 3 > 2 = 2 v(2(1)),
  \end{equation*}
  applying Hensel's lemma to \(f(x) = x^{2} - u \), we see that \(u\) is a square in $L_v$, say \(u = a^{2}\). Hence, 
  \begin{equation*}
    a^{2} = u(1)^2 + b(0)^{2}
  \end{equation*}
  so \((u, b)_v = 1\). 
  
  Now let $k_v$ be the residue field of $L_v$ and let $u = 1 - 4 \zeta_r$. Since $k_v^\times$ is odd, we can solve the equation $b x^2 = \zeta_r$ in $k_v^\times$ for $x \in k_v^\times$. Then
  $u(1)^2 + b (2x)^2 \equiv 1 \pmod 8$ and hence $u(1)^2 + b (2x)^2 = a^2$ for some $a  \in L_v$ by the argument above. It follows then that $(u,b)_v = (1 - 4 \zeta_r, b) = 1$.
\end{proof}


The following is a strengthened version of Theorem~\ref{thm:sqrt5recip}, now with no condition on $c$ being coprime to $2 \calO_K$.
\begin{theorem}
\label{reciprocity-main}
Let \(K \) be a number field and let \(L = K(\zeta_{r})\) where \(\zeta_{r}\) be a primitive $r$th root of unity such that $(r,p) = 1$ and $r \not= 2$.
If $r = 1$, assume further that $K$ has an even number of real embeddings.

Suppose $L$ is unramified at $2$, $K$ is totally split at $2$,  $(p)$ is coprime to $(1 - 4 \zeta_r)$, and write $\zeta_{r} = \zeta_r'^p$ where $\zeta_r'$ is a primitive $r$th root of unity. 

If $(a,b,c)$ is a primitive solution over $\calO_K$ to
\begin{equation*}
    a^{p} + b^{p} + c^{p} = 0
\end{equation*}
such that 
\begin{equation}
\label{unity-condition}
  (ab)^p - \zeta_r c^{2p} \not= 0,
\end{equation}  
then one has
\begin{equation}
\label{jacobi}
				\legendreL{ab - c^{2}\zeta'_{r}}{1 - 4\zeta_{r}}
				\neq -1
\end{equation}
for all permutations of $a,b,c$.
\end{theorem}
\begin{proof}
Since $1 - 4 \zeta_r$ is coprime to $2 \calO_L$, we have from the definition of Jacobi (see \cite[Definition 8.2]{neukirch} for instance) symbol that
\begin{equation}
    \legendreL{ab - c^{2}\zeta'_{r}}{1 - 4\zeta_{r}} = \begin{cases}
        0 & \text{ if } (ab - c^2 \zeta_r') \text{ is not coprime to } (1  - 4 \zeta_r), \\
        \prod_{v \in \oddS{L}, v(1 - 4 \zeta_r) \text{ odd}} \legendre{ab-c^2 \zeta_r'}{\pi_v}^{v(1-4\zeta_r)} & \text{ otherwise},
    \end{cases}
\end{equation}
noting the product is by definition over the $v \in \oddS{L}$, but only the terms  with $v(1-4\zeta_r)$ odd can be $-1$.

Using Lemma~\ref{hilbert-odd}, we see that
\begin{equation*}
    \legendre{ab-c^2 \zeta_r'}{\pi_v}^{v(1-4\zeta_r)} = (1 - 4 \zeta_r, ab-c^2 \zeta_r')_v
\end{equation*}
if $v \in \oddS{L}$, $v(1-4\zeta_r)$ is odd, and $v(ab-c^2\zeta_r') = 0$. Hence, we obtain that
\begin{equation}
    \legendreL{ab - c^{2}\zeta'_{r}}{1 - 4\zeta_{r}} = \begin{cases}
        0 & \text{ if } (ab - c^2 \zeta_r') \text{ is not coprime to } (1  - 4 \zeta_r), \\
        \prod_{v \in \oddS{L}, v(1-4\zeta_r) \text{ odd}} (1 - 4 \zeta_r, ab - c^2 \zeta_r')_v & \text{ otherwise.}
    \end{cases}
\end{equation}

As $(a,b,c)$ is primitive, for any $v \in \evenS{L}$, two of the elements $a,b,c$ are units in $\calO_{L,v}$ and remaining element has positive valuation, since $a, b, c \in \calO_K$ and $K$ is totally split at $2$. This implies that
\begin{equation*}
   ab - c^2 \zeta_r'
\end{equation*}
is coprime to $2 \calO_L$. By Lemma~\ref{hilbert-even}, it follows that
\begin{equation*}
    (1 - 4 \zeta_r, ab - c^2 \zeta_r')_v = 1
\end{equation*}
for all $v \in \evenS{L}$ as we are assuming $L$ is unramified at $2$.

If $r \ge 3$, then $L$ is totally complex and we have that $(1-4 \zeta_r, ab - c^2 \zeta_r')_v = 1$ for all $v \in S_L^\infty$.

If $r = 1$, then $L = K$ is a number field with an even number of real embeddings by hypothesis. Thus, using Lemma~\ref{fermat-real-r1}, we obtain that $$\prod_{v \in S_K^\infty} (-3, ab-c^2)_v = 1$$
as the Hilbert symbols which are $-1$ occur in total an even number of times in the above product.

Thus, using Theorem~\ref{thm:hilbert-constraint}, we see the third product in \eqref{reciprocity-product} is $1$, which implies that
$$\legendreL{ab - c^{2}\zeta'_{r}}{1 - 4\zeta_{r}} = 1.$$
The hypothesis \eqref{unity-condition} ensures the quantity in \eqref{hilbert-identity} is non-zero.

\end{proof}

\begin{remark}
 The constraint in \eqref{reciprocity-product} of Theorem~\ref{thm:hilbert-constraint} is an example of a Brauer-Manin obstruction. Let $\mathcal{A}$ be the Azumaya algebra given by the Hilbert symbol $(t, s (C - \zeta_r' B^2))$ on the variety defined by \eqref{hilbert-identity}, subject to $A, B, C$ coprime and $C^n - \zeta_r B^{2n} \not= 0$. Then \eqref{use-hilbert-reciprocity} which is used in the proof of \eqref{reciprocity-product} corresponds to the condition
 \begin{equation*}
     \sum_v \text{inv}_v \mathcal{A}(x_v) = 0
 \end{equation*}
 of the first formula in the introduction of \cite{bright}.
\end{remark}

\section{Essential obstructive S-unit solutions}

Let $K = \Q(\sqrt{5})$ or $K = \Q(\sqrt{17})$ and $S$ be the set of primes of $K$ above $2$.

Suppose $(a,b,c) \in \calO_K^3$ is a non-trivial primitive solution to
\begin{equation}
\label{abc-equ}
    a + b + c = 0.
\end{equation}
If the Frey curve 
\begin{equation*}
    E_{a,b,c} : y^2 = x(x-a)(x+b)
\end{equation*}
is an elliptic curve over $K$ with conductor $N = 2^3 \calO_K$ for $K = \Q(\sqrt{5})$ and conductor $N = 2 \calO_K$ for $K = \Q(\sqrt{17})$, then the triple $(a,b,c)$ poses an obstruction to the modular method for solving FLT over $K$. If we have such a triple $(a,b,c)$, then multiplication by the square of a unit in $\calO_K$ produces another such triple with $E_{a,b,c}$ in the same isomorphism class over $K$.

Given such a triple $(a,b,c)$ as above, it can be seen that division by any one of $a,b,c$ in \eqref{abc-equ} gives a solution to the $S$-unit equation
\begin{equation}
\label{S-unit-equ}
    U + V = 1
\end{equation}
where $U, V \in K$ and $v(U), v(V) = 0$ for all $v \notin S$. However, a solution to the $S$-unit equation may not give rise to a triple $(a,b,c)$ whose associated $E_{a,b,c}$ has the required conductor $N$, so not all $S$-unit solutions are relevant as obstructions to the modular method.

In this section, we list the $S$-unit solutions which give rise to elliptic curves over $K$ with conductor $N$. This serves two purposes: firstly it gives a more detailed description of the obstructions to solving FLT over these quadratic fields, and secondly it gives a double check on the computations used to prove the main results of this paper.

The direct computation of $S$-units for $\Q(\sqrt{5})$ can be done in a few minutes in ${\tt SageMath}$. However, a direct computation of $S$-units for $\Q(\sqrt{17})$ using {\tt SageMath} is slow and does not appear to terminate in any reasonable amount of time. We explain instead an alternate method of determining these particular $S$-unit solutions using known lists of elliptic curves over $K$ with these conductors \cite{lmfdb}.

Firstly, recall if we have two elliptic curves $E_1$ and $E_2$ over $K$,  given in the following form,
\begin{align*}
    E_1 : y_1^2 = f_1(x_1), \\
    E_2 : y_2^2 = f_2(x_2),
\end{align*}
where $f_i \in \calO_K[x]$ are monic of degree $3$, then $E_1$ is isomorphic to $E_2$ over $K$ if and only if there exist $u, \beta \in \calO_K$ and a change of variables
\begin{align}
    \label{transform-1} y_2 & = u^3 y_2 \\
    \label{transform-2} x_2 & = u^2 x_1 + \beta \\
    \label{transform-3} \Delta(E_2) & = u^{12} \Delta(E_1)
\end{align}

Let $(a,b,c) \in \calO_K^3$ be a non-trivial primitive solution to 
\begin{equation}
    a + b + c = 0.
\end{equation}
Then the `essential' obstructive solutions $(a,b,c)$ we are seeking have the property that the Frey elliptic curve
\begin{equation}
    \label{frey-curve}
    E_0 = E_{a,b,c} : y^2 = x(x-a)(x+b)
\end{equation}
is isomorphic to one of the finitely many elliptic curves $E$ defined over $K$ of conductor equal to the possible Artin conductors $N$ of the residual representations $\rhobar_{E_0,p}$. Due to primitivity, \eqref{frey-curve} is semi-stable at primes $\Fq$ of $K$ coprime to $2 \calO_K$, and hence already a local minimal Weierstrass model at such $\Fq$.

Suppose $E$ is first put into global minimal Weierstrass form
\begin{equation}
    E: y^2 + a_1 xy + a_3 y = x^3 + a_2 x^2 + x_4 x + a_6. 
\end{equation}
By completing the square, we may transform to a model for $E$ of the form
\begin{equation}
    E : y^2 = f(x)
\end{equation}
where $f \in \calO_K[x]$ is monic of degree $3$, and this model is a local minimal Weierstrass model at all primes $\Fq$ of $K$ coprime to $2 \calO_K$. Let $e_1, e_2, e_3$ be the roots in $K$ of $f$.

By \eqref{transform-3}, we deduce that $u$ is a $S$-unit in $\calO_K$, by definition of local minimal Weierstrass model (i.e.\ a Weierstrass model over $\calO_{K_\Fq}$ whose discriminant has minimal valuation). Also, by \eqref{transform-2}, we see that
\begin{align*}
    & u^2 e_1 + \beta = 0 \\
    & u^2 e_2 + \beta = a \\
    & u^2 e_3 + \beta = -b,
\end{align*}
or equivalently
\begin{align}
  a & =  u^2 (e_2 - e_1) \\
  b & = -u^2 (e_3 - e_1).
\end{align}
Set $c = - a - b$. We find an obstructive solution $(a,b,c)$ if and only there is a $S$-unit $u$ of $K$ such that $a$ and $b$ are coprime ($a,b,c$ mutually coprime implies $a,b,c$ pairwise coprime). If we find such a triple $(a,b,c)$, it is unique up to multiplication by the square of a unit in $\calO_K$.

We have determined the list of `essential' obstructive solutions for both $\Q(\sqrt{5})$ and $\Q(\sqrt{17})$, which can be found in the electronic resources for this paper \cite{programs}. For those congruence classes of exponents $p$ where we obtain a result, all obstructive solutions are eliminated by the reciprocity constraints. For those congruence classes of exponents $p$ where we don't obtain a result, there is some obstructive solution that is not eliminated by the reciprocity constraints.

\bibliographystyle{acm}
\bibliography{main}
\end{document}